\documentclass[12pt]{amsart}
\usepackage{graphicx} 

\title{Stratification of Derived Categories of Tate Motives}
\author{David Rubinstein}
\date{}
\usepackage{amsmath, amssymb, amsthm, latexsym}
\usepackage{ marvosym }
\usepackage{quiver, tikz-cd}
\usepackage{diagrams}
\usepackage{euler}
\usepackage[mathscr]{eucal}
\usepackage{wasysym}
\usepackage{todonotes}
\usepackage{bbm}
\usepackage{bbold}
\usepackage{lmodern}
\usepackage{hyperref}
\DeclareMathAlphabet{\mathbx}{U}{BOONDOX-ds}{m}{n}
\SetMathAlphabet{\mathbx}{bold}{U}{BOONDOX-ds}{b}{n}
\DeclareMathAlphabet{\mathbbx} {U}{BOONDOX-ds}{b}{n}
\usepackage[all]{xy}
\CompileMatrices 
\usetikzlibrary{arrows}
\newcommand{\nc}{\newcommand}
\nc{\dmo}{\DeclareMathOperator}
\dmo{\SH}{SH}
\nc{\SHc}{{\SH^c}}
\dmo{\Spec}{Spec}

\newcommand{\ZZ}{\mathbx Z}
\newcommand{\QQ}{\mathbx Q}

\newcommand{\MM}{\mathscr M}

\newcommand{\ol}{\overline}
\newcommand{\olQ}{\ol{\QQ}}

\newcommand{\mc}{\mathcal}

\newcommand{\unit}{\mathbb{1}}
\newcommand{\heart}{\mc T^\text{\Heart}}

\newcommand{\lt}{Locid}
\newcommand{\eM}{g({\MM})}

\newcommand{\AT}{\mc A \mc T}

\newcommand{\DMe}{DM^{\text{ét}}}
\newcommand{\DTMe}{DTM^{\text{ét}}}
\newcommand{\dtmp}{DTM(\olQ, \ZZ_{(p)})}
\newcommand{\dtmep}{DTM^{\text{ét}}(\olQ, \ZZ_{(p)})}
\newcommand{\dtmq}{DTM(\olQ, \QQ)}
\newcommand{\dtmz}{DTM(\olQ, \ZZ)}
\newcommand{\dtmmodp}{DTM(\olQ, \ZZ/p)}
\newcommand{\dtmemodp}{DTM^{\text{ét}}(\olQ, \ZZ/p)}

\theoremstyle{plain}
\newtheorem{theorem}{Theorem}[section]
\newtheorem{corollary}[theorem]{Corollary}
\newtheorem{lemma}[theorem]{Lemma}
\newtheorem{proposition}[theorem]{Proposition}

\theoremstyle{definition}

\newtheorem{definition}[theorem]{Definition}
\newtheorem{remark}[theorem]{Remark}

\newtheorem{question}[theorem]{Question}
\newtheorem{hypothesis}[theorem]{Hypothesis}

\begin{document}
\begin{abstract}
   We classify the localizing tensor ideals of the derived categories of mixed Tate motives over certain algebraically closed fields. More precisely, we prove that these categories are stratified in the sense of Barthel, Heard and Sanders. A key ingredient in the proof is the development of a new technique for transporting stratification between categories by means of Brown--Adams representability, which may be of independent interest.  
\end{abstract}
\maketitle

\tableofcontents
\section{Introduction}
The theory of stratification for a rigidly-compactly generated tensor triangulated category \(\mc T\) developed by Barthel, Heard and Sanders \cite{BHSStrat} provides a framework for classifying the localizing ideals of ~\(\mc T\).\hspace{.1cm} At its heart is a support theory \(Supp\) for \(\mc T\) that takes value in the Balmer spectrum of compact objects \(Spc(\mc T^c)\). This support theory is given by tensoring with certain tensor idempotent objects \(g(\mc P) \in \mc T\) associated to each \(\mc P \in Spc(\mc T^c)\). Explicitly, we have \[Supp(t) = \{\mc P \in Spc(\mc T^c): t \otimes g(\mc P) \ne 0\}\]
for each \(t \in \mc T\). The category \(\mc T\) is \textbf{\textit{stratified}} if this support theory provides an order-preserving bijection  between the collection of localizing tensor ideals of \(\mc T\) and the collection of all subsets of \(Spc(\mc T^c)\). 
In this paper, we prove that the derived category of mixed Tate motives over certain algebraically closed fields is stratified. Our starting point is the computation of their Balmer spectra by Gallauer \cite{martinspecpaper}:
\begin{theorem}[Gallauer]
   The Balmer spectrum of \(DTM(\olQ, \ZZ)^c\) is the following topological space: 
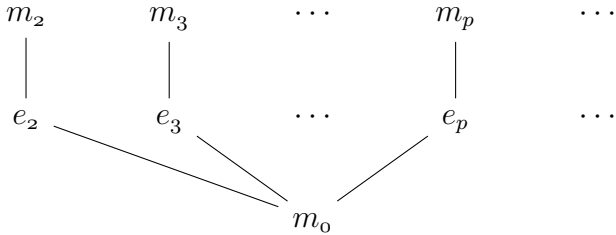
\[\begin{tikzcd}
	& {m_2} & {m_3} & \cdots & {m_p} & \cdots \\
	& {e_2} & {e_3} & \cdots & {e_p} & \cdots \\
	&&& {m_0}
	\arrow[no head, from=3-4, to=2-2]
	\arrow[no head, from=2-3, to=3-4]
	\arrow[no head, from=3-4, to=2-5]
	\arrow[no head, from=2-5, to=1-5]
	\arrow[no head, from=1-3, to=2-3]
	\arrow[no head, from=1-2, to=2-2]
\end{tikzcd}\]
\end{theorem}
We will explain what these primes are in Theorem~\ref{maintheoremmartin}. Our main theorem is the following (Theorem~\ref{maintheoremdtmstrat}):
\begin{theorem}
    The category \(DTM(\olQ, \ZZ)\) is stratified.
\end{theorem}
There are many consequences of a category \(\mc T\) being stratified (see, e.g. \cite[Part~III]{BHSStrat}). For example, if \(\mc T\) is stratified and \(Spc(\mc T^c)\) satisfies a mild topological assumption, then the abstract telescope conjecture holds (see \cite[Theorem~9.11]{BHSStrat}). This topological condition holds for \(Spc(DTM(\olQ, \ZZ)^c)\) for example, so as a corollary to the above theorem, we get (Theorem~\ref{telescopeconjecture}):
\begin{theorem}
    Every smashing tensor ideal of \(DTM(\olQ, \ZZ)\) is compactly generated.
\end{theorem}
Proving that a category is stratified amounts to establishing two properties: (a) the `local-to-global' principle, which morally says every object can be reconstructed from its local pieces; and (b), the `minimality' property, which says the localizing ideals generated by the tensor idempotents \(g(\mc P)\) are minimal. Importantly, Barthel, Heard and Sanders, building off the earlier work of Stevenson, prove that if the spectrum is noetherian, as the above spectrum of Tate motives is for example, the local-to-global principle always holds, and so the problem reduces to proving the minimality property. Moreover, minimality is a local condition, so the bulk of this paper is devoted to proving that minimality holds at the unique closed points in the corresponding local categories for each prime in \( Spc(DTM(\olQ, \ZZ)^c)\). While there are infinitely many primes in \( Spc(DTM(\olQ, \ZZ)^c)\), there are qualitatively only 3 distinct primes to consider: \(m_0, e_p, \text{ and } m_p\).

The local category associated to the generic point \(m_0\) is equivalent to the category of rational motives \(\dtmq\). We prove this category is stratified in \S\ref{m_0section} using an extension of work on t-structures for rational motives by Peter and Levine \cite{peterspecQ, tatemotiveststructure}. 

For the heights one and two primes, \(e_p \text{ and }m_p\), we use a result similar in spirit to the quasi-finite descent of \cite{descent} to reduce to the following two local categories: \begin{enumerate}
    \item \(\dtmemodp\): which corresponds to \(e_p\) (see \S\ref{e_psection}), and
    \item \(\dtmmodp\): which corresponds to \(m_p\) (see \S\ref{m_psection}).
\end{enumerate} Stratification of \(\dtmemodp\) is an immediate consequence of the so-called `Rigidity Theorem' for étale motives with finite coefficients (see Theorem~\ref{rigiditytheorem}). On the other hand, in proving \(\dtmmodp\) is stratified in \S\ref{m_psection}, we are led naturally to the following question, which may be of independent interest. 
\begin{question}\label{questioninintro}
   Suppose \(\mc T_1\) and \(\mc T_2\) are two rigidly-compactly generated tensor-triangulated categories with a tt-equivalence between their compact parts~\(\mc T_1^c \simeq \mc T_2^c\). If \(\mc T_1\) is stratified, does it follow that \(\mc T_2\) is stratified as well?
\end{question} 
In \S\ref{brownadamssection}  we investigate this question through the lens of Brown--Adams representability. More precisely, given a rigidly-compactly generated tt-category \(\mc T\), we study the module category \(\mc A := Add((\mc T^c)^{op}, \mc Ab)\) along with the restricted Yoneda functor \begin{align*}
    h: \hspace{.1cm}&\mc T \xrightarrow{}  \mc A \\
    & t \xrightarrow{} \hat{t}: = Hom(-, t)|_{\mc T^c}.
\end{align*}
The essential image of restricted Yoneda lies in the subcategory of homological functors, and in certain examples a much stronger relationship holds. For example, Adams showed in \cite{OGbrownadamrepinSH} that for the stable homotopy category \( \mc T = \text{SH}\), every homological functor \(\mc H: (\mc T^c)^{op} \to \mc{A}b\) is the restriction to \(\mc T^c\) of a representable functor on \(\mc T\), and every map between homological functors is induced by a (non-unique) map between the representing spectra. Due to this historical example, a tt-category is said to satisfy Brown--Adams representability if it satisfies those above two properties. Determining whether or not a category satisfies Brown--Adams representability is a challenging problem in general. Importantly for us however, Neeman, generalizing a theorem of Brown, showed in \cite{onatheoremofbrownandadams} that a sufficient condition is that the subcategory of compact objects \(\mc T^c\) is equivalent to a countable category. 

The category of modules \(\mc A\), and its relation to \(\mc T\), has also been studied by Balmer, Krause and Stevenson in a series of recent papers \cite{ttfieldsrumination, ttfieldframesmashing, homsupport, ttfieldnilpotence}. In particular, they use \(\mc A\) to define a new spectrum, the so-called homological spectrum \(Spc^h(\mc T^c)\). Associated to each `homological prime' \(\beta \in Spc^h(\mc T^c)\) is a certain pure-injective object \(E_\beta \in \mc T\), and these can be used to define a new support theory for \(\mc T\). These \(E_\beta\) objects are often much better behaved than the analogous \(g(\mc P)\) objects are. For example, they are often field objects (see Remark~\ref{fieldEbobjects}). Our main theorem for this section is the following (see Theorem~\ref{stratifiedusingbrownadamstheorem}):
\begin{theorem}\label{intro brown adams theorem}
    Suppose \(\mc T_1\) and \(\mc T_2\) are two rigidly-compactly generated tensor-triangulated categories with a tt-equivalence between their compact parts \(\mc T_1^c \simeq \mc T_2^c\). Suppose further that: \begin{enumerate}
        \item [(1)] \(\mc T_1\) is stratified;
        \item [(2)] \(\mc T_1^c\) is equivalent to a countable category;
        \item [(3)] The \(E_\beta\) objects in \(\mc T_1\) are field objects; and
        \item [(4)] For every nonzero homological functor \(\hat{t} \in \mc A_1\) there exists a nonzero map \(\hat{E}_\beta \otimes \hat{x} \to \hat{t}\) for some \(\beta \in Spc^h(\mc T_1^c)\) and compact \(x \in \mc T_1^c\).
    \end{enumerate}  Then \(\mc T_2\) is also stratified.
\end{theorem}     
As mentioned above, we use this theorem to prove that \(\dtmmodp\) is stratified. The compact part of this category is known to be equivalent to the (bounded) derived category of filtered \(\ZZ/p\) vector spaces (see Chapter \ref{m_psection}), and this category is well understood and amenable to the techniques of stratification. In particular, it satisfies the hypothesis of the above theorem. This allows us to conclude that \(\dtmmodp\) is stratified (see Theorem~\ref{dtmmodpstrattheorem}), and hence deduce that minimality holds at the height two prime \(m_p\).

A final comment to make is that, while we are only considering motives over \(\olQ\) in this paper, our results hold for other algebraically closed fields satisfying a certain vanishing condition, as in \cite[Hypothesis~6.6]{martinspecpaper}. In particular, this vanishing condition holds for \(\overline{\mathbx F}_p\), see \cite[Remark~6.8]{martinspecpaper}. Moreover, our results are conjectured to hold for \(\mathbx C\). More specifically, the conjecture is as follows: for any real closed field \(\mathbx F\), there is a canonical surjective comparison map (see \cite[\S5]{SpectraSpectraSpectra}) \[\rho: Spc(DTM(\mathbx F, \ZZ)^c) \to Spec(\ZZ)\]
and so the Balmer spectrum can be found by identifying the fibers of this map. For fields satisfying the vanishing hypothesis, the fiber of the generic point \((0)\) is known to be a singleton (for example, the fiber of \((0)\) is \(m_0\) for the case \(DTM(\olQ,\ZZ)^c\)), whereas the fiber is only conjecturally a singleton for \(\mathbx F = \mathbx C\) (see \cite[Remark~11.8]{threerealartintatemotives}). If the fiber is a singleton as conjectured, then our proofs hold in their entirely for \(\mathbx C\) as well.

\section{Balmer--Favi Support and Stratification}
Throughout the paper \(\mc T\) will denote a rigidly-compactly generated tensor-triangulated category and we will write \(\mc T^c\) for the subcategory of compact-rigid objects. We assume that the reader has a general understanding of tt-geometry as in \cite{balmer2005} and the theory of stratification as in \cite{BHSStrat}. Let us briefly remind the reader of the relevant definitions and results we shall use for this paper. 
\begin{definition}\label{weaklynoeth}
Let \(\mc P \in Spc(\mc T^c)\) be a prime. We say \(\mc P\) is \textbf{weakly visible} if \(\{\mc P\}\) is the intersection of a Thomason subset and the complement of a Thomason subset. We say \(\mc P\) is \textbf{visible} if its closure is Thomason. Note that a visible prime is weakly visible. If every prime \(\mc P \in Spc(\mc T^c)\) is weakly visible (respectively visible) we say \(Spc(\mc T^c)\) is weakly noetherian (respectively noetherian). 
\end{definition}
Associated to each weakly visible prime \(\mc P \in Spc(\mc T^c)\) is a nonzero tensor idempotent object \(g(\mc P)\) in \(\mc T\) \cite{balmerfavisupport, BHSStrat}. These objects are used to define a support theory on \(\mc T\):
\begin{definition}[Balmer--Favi]
Suppose that \(Spc(\mc T^c)\) is weakly noetherian. The Balmer--Favi support of an object \(t \in \mc T\) is \[Supp(t): = \{\mc P \in Spc(\mc T^c) : t \otimes g(\mc P) \ne 0\}.\]
\end{definition}
\begin{remark} \label{ g(P) objects in two-point case}
    Let us make a quick remark about these \(g(\mc P)\) objects in a special case. Let \(\mc T\) be a rigidly-compactly generated tt-category, and suppose \(Spc(\mc T^c)\) consists of two, connected points \[\begin{tikzcd}
	{Spc(\mc T^c)} \\
	\MM \\
	\mc P.
	\arrow[no head, from=2-1, to=3-1]
\end{tikzcd}\]
That is, a closed point \(\mc M\) and a generic point \(\mc P\). Consider the idempotent triangle \[e_{\MM} \to \unit \to f_{\MM} \to \Sigma e_{\MM}\]
associated to the finite localization for \(Y = \{\mc M\}\). We have that \begin{enumerate}
    \item [(1)] \(g(\MM) = e_{\MM}\), and 
    \item [(2)] \(g(\mc P) = f_{\MM}\).
\end{enumerate}
We will use this observation later in \S\ref{m_psection}.
\end{remark}
\begin{remark}
    The Balmer-Favi support extends the universal Balmer support for compact objects. That is, for \(x \in \mc T^c\) we have \(Supp(x) = supp(x)\); see \cite[Lemma~2.18]{BHSStrat}.
\end{remark} 
\begin{definition}[Local-to-Global Principle]
Suppose that \(Spc(\mc T^c)\) is weakly noetherian. We say \(\mc T \) satisfies the \emph{local-to-global principle} if \[Locid(t) = Locid(t \otimes g(\mc P) : \mc P \in Spc(\mc T^c))\]
for any \(t \in \mc T\).
\end{definition}
\begin{remark}\label{detectionprop}
Note that if \(\mc T\) satisfies the local-to-global principle, then it also satisfies the so-called detection property: \[Supp(t) = \varnothing \iff t=0. \]
\end{remark}
\begin{remark}\label{localglobalnoeth}
     While there is no general characterization for when the local-to-global principle holds, a sufficient condition is that \(Spc(\mc T^c)\) is noetherian. This is established in \cite[Theorem~3.22]{BHSStrat}, building off the work of Stevenson \cite{supporttheoryviaactions, localtoglobalstevenson}.
\end{remark}
Now we come to the main characterization of stratification established in \cite[Theorem~4.1]{BHSStrat}:
\begin{theorem}[Stratification]\label{strattheorem}
Suppose \(Spc(\mc T^c)\) is weakly noetherian. Then the following are equivalent: \begin{enumerate}
    \item [(a)] The local-to-global principle holds for \(\mc T\) and for each prime ~\(\mc P \in Spc(\mc T^c)\) we have that \(Locid(g(\mc P))\) is a minimal localizing tensor ideal of \(\mc T\).
    \item [(b)] For all \(t \in \mc T\), \(Locid(t) = Locid ( g(\mc P) : \mc P \in Supp(t)) \).
    \item [(c)] The map \begin{align*}
        \{\text{Localizing ideals of } \mc T\} &\xrightarrow{Supp} \{\text{Subsets of } Spc(\mc T^c)\}\\
        \mc L  \mapsto Su & pp(\mc L)  := \bigcup_{t \in \mc L}Supp(t)
    \end{align*}is a bijection.
\end{enumerate}
\end{theorem}
\begin{definition}
We say that \(\mc T \) is \textbf{stratified} if any of the equivalent conditions (a)-(c) from above hold. We will say that \(\mc T \) has ``\textbf{minimality at \(\mc P\)}", or satisfies the minimality property at \(\mc P\), if \(\lt (g(\mc P))\) is a minimal localizing ideal. Importantly, the minimality property is a condition that can be checked locally:
\end{definition}
\begin{theorem}\cite[Proposition~5.2]{BHSStrat} \label{localresult}
Suppose \(\mc T\) satisfies the local-to-global principle and has weakly noetherian spectrum. Then \(\mc T \) has minimality at \(\mc P \in Spc(\mc T^c)\) if and only if the local category \(\mc T / \lt (\mc P)\) has minimality at its unique closed point.
\end{theorem}
\begin{remark}
Thus, assuming the local-to-global principle holds, to establish \(\mc T\) is stratified it suffices to check if each local category \(\mc T / \lt (\mc P)\) has minimality at its unique closed point.
\end{remark}
\begin{remark}
    The theory of stratification is amenable to various descent techniques (see, for example \cite{descent, cosupport}). Let us close out this section with a slight modification of `quasi-finite descent' \cite[Theorem~17.16]{cosupport} that we will use throughout the paper.
\end{remark}
\begin{definition} \label{closed functor}
     Let \(\mc T\) and \(\mc S\) be rigidly-compactly generated tt-categories. A coproduct preserving tt-functor \(f^*: \mc{T} \to \mc{S}\) is called a \textbf{geometric functor}. We denote the right adjoint by \(f_*\) (which exists by Brown--Neeman Representability, see for example \cite{GrothNeemanReppaper}).
\end{definition}
\begin{proposition}\label{compacttheorem} Let \( f^*: \mc T \to \mc S\)  be a geometric functor and let \(\varphi:=Spc(f^*) : Spc(\mc S^c) \to Spc(\mc T^c)\) denote the induced map on spectra. Assume that: \begin{enumerate}
    \item \(\mc T\) is local and satisfies the detection property;
    \item \(f_*(\unit_{\mc S}) \in \mc T\) is compact;
    \item There exists a unique \(\mc P \in Spc(\mathcal{S}^c)\) such that \(\varphi(\mc P)= \MM\) is the unique closed point in \(Spc(\mc T^c)\).
\end{enumerate}
Then minimality at \(\mc P\) implies minimality at \(\MM\). 
\end{proposition}
\begin{proof}
Our first step will be to show that \(f^*\) is conservative on all objects supported on \(\MM\). Let \(Supp(t) \subseteq \{\MM\)\} and suppose \(f^*(t) = 0\). Then \(0=f_*f^*(t) \simeq f_*(\unit_{\mc S}) \otimes t \). Hence \[\varnothing = Supp(f_*(\unit) \otimes t)= supp(f_*(\unit_{\mc S})) \cap Supp(t) \]
where the second equality holds from the half-tensor product formula (see \cite[Lemma~2.18]{BHSStrat}). Since \(\mc S \ne 0\) we have that \(\unit_{\mc S} = f^*(\unit_{\mc T}) \ne 0\) and so \(f_*(\unit_{\mc S}) \ne 0\) since \(f_*\) is conservative on the image of \(f^*\) (see \cite[Remark~13.9]{cosupport}). Now \(supp(f_*(\unit_{\mc S}))\) is closed since we are assuming \(f_*(\unit_{\mc S})\) is compact, and so it contains the unique closed point \(\MM\), which forces \(Supp(t)=\varnothing\). Hence we conclude that \(t=0\) since \(\mc T\) satisfies the detection property. This establishes that \(f^*\) is conservative on on all objects supported on \(\MM\). Now suppose \(0 \ne t \in \lt(\eM) \) and note that \[0 \ne f^*(t) \in \lt(f^*(\eM))= \lt (g(\mc P))\] where \(f^*(\eM) = g(\varphi^{-1}(\{\MM\}) ) = g(\mc P)\) by \cite[Proposition~5.11]{spectrumofequivstablehomotopy} and by the hypothesis on the pre-image of \(\{\MM\}\). Since we are assuming minimality at \(\mc P\) we have an equality \[\lt(f^*(t))= \lt(f^*(g(\MM))).\] 
Hence we get \(f_*f^*(\eM) \in f_*(\lt(f^*(t) )) \subseteq \lt (t)\) by \cite[13.4]{cosupport}. Finally we compute: \begin{align*}
    \eM &\in \lt(\eM \otimes f_*(\unit)) \text{   (by \cite[Lemma~3.7]{BHSStrat})}\\
    & = \lt (f_*f^*(\eM))\\
    & \subseteq \lt(t)
\end{align*}
which establishes minimality at the prime \(\MM\) as desired.
\end{proof}

\section{Derived Categories of Motives}\label{motivesinfosection}
In this section we recall the computation of the Balmer spectra of certain motivic categories by Martin Gallauer. We will follow \cite{motivesbook, etalemotives} for the general theory of motives, and will also refer heavily to \cite{martinspecpaper} and the references therein. 
\begin{definition}
    We fix a commutative ring \(R\), and the field \(\olQ\). The derived category of motives over \(\olQ\) is the tt-category, denoted \(DM(\olQ, R)\), that is constructed out of the derived category of Nisnevich sheaves with transfers of \(R\)--modules on the category of smooth, finite type \(\olQ\)--schemes; see \cite[11.1.1, 11.1.2]{motivesbook} for further details.
\end{definition}
\begin{remark}
    This category comes equipped with a functor \[R(-): Sm/\olQ \to DM(\olQ, R)\] that sends a smooth \(\olQ\)-scheme \(X\) to its underlying `motive', which is denoted \(R(X)\). Moreover, the construction of \(DM(\olQ, R)\) inverts the motive of the projective line, which we denote by \(R(1)\). For any integer~\(n\), we then denote the `Tate-twist' of weight \(n\) by \(R(n): = R(1)^{\otimes n}\), and one gets that \(R(i) \otimes R(j) = R(i+j)\) for all \(i, j \in \ZZ\).
\end{remark} 
\begin{definition}
    The category of Tate motives, denoted~\(DTM(\olQ, R)\), is defined as \(DTM(\olQ, R) = Loc(R(n): n \in \ZZ)\). It is a rigidly-compactly generated tensor-triangulated category; see for example \cite[Example~5.17]{charofetalmorphism} and the references therein. 
\end{definition}
\begin{remark}\label{if R' in D(R)^c then strongly closed}
    Given a ring homomorphism \(R \to R'\) there is an induced geometric functor \(\gamma^*: DM(\olQ, R) \to DM(\olQ, R') \) 
which restricts to a geometric functor on Tate motives  \(\gamma^*: DTM(\olQ, R) \to DTM(\olQ, R') \). 
The right adjoint \(\gamma_*\) of this geometric functor is always conservative (see \cite[5.4.2]{etalemotives}). Moreover, if \(R' \in D(R)^c\), then \(\gamma_*\) preserves compact objects \cite[\S3]{martinspecpaper}.
\end{remark}
There is a similar, parallel story to tell about the so-called étale motives:  
\begin{definition}
    We again fix a commutative ring \(R\), and the field \(\olQ\). The derived category of \text{étale} motives over \(\olQ\) is the tt-category denoted \(DM^{\text{ét}}(\olQ, R)\), that is constructed out of the derived category of \text{étale} sheaves with transfers of \(R\)--modules on the category of smooth , finite type \(\olQ\) schemes; see \cite{etalemotives} and \cite[\S4]{martinspecpaper} for further details.
\end{definition}
This category again comes equipped with a `\text{étale}' motive functor \[R^{\text{ét}}(-): Sm/\olQ \to \DMe(\olQ, R)\]
and we again have invertible étale Tate objects \(R^{\text{ét}}(n)\) arising from the étale motive of the projective line.
\begin{definition}
    The derived category of \text{étale} Tate motives is the tt-category \(\DTMe(\olQ,R)\) = \(Loc(R^{\text{ét}}(n): n \in \ZZ)\). When \(R\) is any localization or quotient of \(\ZZ\), this category is also rigidly-compactly generated (see \cite[\S4]{martinspecpaper}).
\end{definition}
\begin{remark}
Given a ring map \(R \to R'\), we again have a geometric functor \(\gamma^*: D(T)M^{\text{ét}}(\olQ, R) \to D(T)M^{\text{ét}}(\olQ, R')\) with conservative right adjoint \(\gamma_*\) which preserves compacts whenever \(R' \in D(R)^c\). 
\end{remark}
\begin{remark} \label{etalification functor}
    Essentially since every Nisnevich sheaf is also an \text{étale} sheaf, there is a canonical `\text{étalification}' tt-functor \[\alpha_{\text{ét}}: DTM(\olQ,R) \to \DTMe(\olQ, R)\]
that commutes with change of coefficients. Moreover, whenever \(\QQ \subseteq R\) this functor is an equivalence of categories \cite[16.1.2]{motivesbook}.
\end{remark}
An important result by Gallauer provides another deep connection between étale Tate motives and Tate motives:
\begin{theorem}\cite[Theorem~C.4]{martinspecpaper}\label{locgivesetal}
Let \(p\) be a prime number and let \(\beta_p: \ZZ/p(0) \to \ZZ/p(1)\) denote the Bott map. Consider the geometric functor \[\gamma^*: \dtmp \to DTM(\olQ, \ZZ/p)\]
induced by the ring map \(\ZZ_{(p)} \to \ZZ/p\) and let \[\gamma_*: DTM(\olQ, \ZZ/p) \to \dtmp \]
denote its right adjoint. Then the étalification functor \[\alpha_{\text{ét}}: DTM(\olQ,\ZZ_{(p)}) \to \DTMe(\olQ, \ZZ_{(p)})\]
is a finite localization with kernel \(\lt (\gamma_*(cone(\beta_p)))\). In particular, we have an equivalence \[\dtmp/ \lt( \gamma_*(cone(\beta_p))) \simeq \dtmep \]
\end{theorem}
Armed with this, we can now describe Gallauer's computations of the Balmer spectra of Tate motives. Let us first state the spectrum for finite coefficients:
\begin{proposition}\cite[Proposition~8.2]{martinspecpaper}\label{specdtmfinite}
Let \(\beta_p: \ZZ/p(0) \to \ZZ/p(1)\) denote the Bott map. Then \(Spc(DTM(\olQ, \ZZ/p)^c)\) consists of two connected points:
\[\begin{tikzcd}[column sep=tiny,row sep=small]
	{Spc(DTM(\olQ, \ZZ/p)^c)}= & (0) \\
	& \langle cone(\beta_p) \rangle 
	\arrow[no head, from=1-2, to=2-2]
\end{tikzcd}\]
\end{proposition}
The last result needed before the main computation by Gallauer is the following computation of the spectrum with rational coefficients in \cite[Theorem~4.15]{peterspecQ}: 
\begin{theorem}[Peter]
The Balmer spectrum of rational motives is \[Spc(DTM(\olQ, \QQ)^c) \simeq Spec(\QQ) \simeq (0)\]
\end{theorem}
We can now bring these results together for integral coefficients \cite[Theorem~8.6]{martinspecpaper}: 
\begin{theorem}[Gallauer] \label{maintheoremmartin}
The Balmer spectrum of \(DTM(\olQ, \ZZ)^c\) is the following noetherian topological space 
\[\begin{tikzcd}
	& {m_2} & {m_3} & \cdots & {m_p} & \cdots \\
	  &{e_2} & {e_3} & \cdots & {e_p} & \cdots \\
	&&& {m_0}
	\arrow[no head, from=3-4, to=2-2]
	\arrow[no head, from=2-3, to=3-4]
	\arrow[no head, from=3-4, to=2-5]
	\arrow[no head, from=2-5, to=1-5]
	\arrow[no head, from=1-3, to=2-3]
	\arrow[no head, from=1-2, to=2-2]
\end{tikzcd}\]
where the \(p\)'s are prime numbers and the specialization relations are depicted by the lines going upwards. The primes are given as follows: \begin{enumerate}
    \item [(height 2)]: \(m_p\) is the kernel of the tt-functor \[\gamma^*_p: DTM(\ol{\QQ}, \ZZ)^c \to DTM(\ol{\QQ}, \ZZ/p\ZZ)^c\] induced from the ring map \(\ZZ \to \ZZ/p\).
    \item [(height 1)]: \(e_p\) is the kernel of the following composite \[DTM(\ol{\QQ}, \ZZ)^c \xrightarrow{\gamma^*_p} DTM(\ol{\QQ}, \ZZ/p\ZZ)^c \xrightarrow{\alpha_{\text{ét}}} DTM^{\text{ét}}(\ol{\QQ}, \ZZ/p\ZZ)^c.\] 
     \item [(height 0)]: \(m_0\) is the kernel of the finite localization \[\gamma^*_{\QQ}: DTM(\ol{\QQ}, \ZZ)^c \to  DTM(\ol{\QQ}, \QQ)^c\] induced from the ring localization \(\ZZ \to \QQ\).
\end{enumerate}
Moreover, the étalification tt-functor \(\alpha_{\text{ét}}: DTM(\olQ,\ZZ) \to \DTMe(\olQ, \ZZ)\) induces an inclusion on spectra \[Spec(\ZZ) \simeq Spc(\DTMe(\olQ, \ZZ)^c) \lhook\joinrel\xrightarrow{Spc(\alpha_{\text{ét}})}  Spc(DTM(\olQ, \ZZ)^c)\] which is a homeomorphism onto the subspace \(\{m_0, e_p: \text{ p prime}\}\).
\end{theorem}
\begin{remark}
    Since \(Spc(DTM(\olQ, \ZZ)^c)\) is noetherian, it automatically satisfies the local-to-global principle (recall Remark \ref{localglobalnoeth}). The rest of this paper will therefore be focused on proving minimality holds at each prime. 
\end{remark}
\begin{remark}\label{3difftypesofprimes}
    We note that, while there are infinitely many primes in ~\(Spc(DTM(\olQ, \ZZ)^c)\), morally there are only 3 `types' of primes. Indeed, each vertical slice in the spectrum is obtained from a finite `algebraic' localization \(DTM(\olQ, \ZZ) \to \dtmp\) (\cite[Example~6.12,~Proposition~10.2]{martinspecpaper}). These are local categories, with \(m_p\) (resp, \(e_p\)) being the \((0)\) ideal now. Pictorially, for example, we have \[\begin{tikzcd}[column sep=tiny,row sep=small]
	& {m_p} = (0) \\
	{Spc(\dtmp^c)}= & {e_p} \\
	& {m_0}
	\arrow[no head, from=1-2, to=2-2]
	\arrow[no head, from=2-2, to=3-2]
\end{tikzcd}\]
Moreover, by \cite[Proposition~1.32]{BHSStrat} in order to prove minimality holds at the primes \(m_p\) and \(e_p\) it suffices to prove it in the local category \(\dtmp\). 
\end{remark}
\section{Minimality at the Height Zero Prime}\label{m_0section}
In this section, we will prove minimality holds at the the height zero prime \(m_0\).
\begin{remark}\label{height0reduction}
    Recall that the height zero prime \(m_0\) is obtained from the algebraic localization going from integral to rational coefficients. That is  \[DTM(\olQ, \ZZ)/\lt (m_0) \simeq DTM(\olQ, \QQ)\]
By Theorem \ref{localresult}, to prove that minimality holds at the prime \(m_0\) it suffices to prove minimality holds at the unique prime in \(DTM(\olQ, \QQ)\). Let us now recall the computation of \(Spc(DTM(\olQ, \QQ)^c)\) by Peter \cite{peterspecQ}. 
\end{remark}
\begin{remark}
    To orient the reader for why we are interested in the following definitions: our goal will be to eventually produce a triangulated equivalence between \(DTM(\olQ, \QQ)\) and the derived category of the heart of a t-structure. Such an equivalence can be given, provided the t-structure satisfies extra properties, which we now explain.
\end{remark}
\begin{definition}\label{tstructure}
Let \(\textbf{t}= (\mc U, \mc V)\) denote a t-structure on a rigidly-compactly generated tt-category \(\mc T\) and let us denote the heart by \(\mc T^\text{\Heart}\). The t-structure \(\textbf{t}\) is said to be: \begin{enumerate}
    \item [(a)] \textbf{Non Degenerate} if \[\bigcap_{k \in \ZZ} \Sigma^k \mc U= 0 = \bigcap_{k \in \ZZ} \Sigma^k \mc V \]
    \item [(b)] \textbf{Smashing} if the coaisle \(\mc V\) is closed under coproducts. 
    \item [(c)] \textbf{Compatible} if \(\heart \otimes \heart \subseteq \heart\) and \(\unit \in \heart\). In this case \(\heart\) is itself a tensor-abelian category, whose tensor structure is exact in both variables. Moreover the unit of \(\heart\) is just the same unit as in \(\mc T\) (see \cite[Remark~3.7]{peterspecQ}).
    \item [(d)] \textbf{Strongly hereditary} if \(hom_{\mc T}(X,\Sigma^i Y) = 0\) for \(X, Y \in \heart\) and \(i \ge 2\). Note that if \(\textbf{t}\) is strongly hereditary then \(\heart\) is a hereditary abelian category, that is \(Ext^i(M,N) = 0\) for all \(i \ge 2\) and \(M,N \in \heart\), see \cite[Proposition~1.6]{tatemotiveststructure}.
\end{enumerate}
\end{definition}
\begin{remark}\label{smashing t struct}
    A t-structure \(\textbf{t}\) is smashing if and only if the cohomology functor \(\mc H^0: \mc T \to \heart\) preserves coproducts, see for example \cite[Lemma~3.3]{smashingtstructure}. 
\end{remark}
\begin{lemma} \label{ext for hereditary cat}
     If a t-structure \(\textbf{t}\) is strongly hereditary, then, for all \(M, N \in \heart\) and for all \(i \in \ZZ\) we have \(Ext^i(M,N) \simeq hom_{\mc T}(M, \Sigma^i N)\).
\end{lemma}
\begin{proof}
    Indeed, for \( i = 0\) we have \(Ext^0(M,N) = hom_{\heart}(M,N) = hom_{\mc T}(M, N)\) since \(\heart\) is a full subcategory. For \(i = 1\) we have \(Ext^1(M,N) = hom_{\mc T}(M, \Sigma N)\) for any heart \(\heart\) (see \cite{perversesheaves}). For \(i \ge 2\) we have \(Ext^i(M,N) = 0 = hom_{\mc T}(M, \Sigma^i N)\) by assumption and since \(\heart\) is a hereditary category. Finally for \(i < 0 \) we always have \(Ext^i(M,N) = 0\) and \(hom_{\mc T}(M, \Sigma^iN) = 0\) since every t-structure \(\textbf{t}\) has this vanishing condition.
\end{proof}
\begin{remark}
    One has the following two inclusion maps \[\begin{tikzcd}
	{\heart} & {\mathcal{T}} \\
	{D(\mathcal{\heart})}
	\arrow["{i_{D(\heart)}}"', hook, from=1-1, to=2-1]
	\arrow["{i_\mathcal{T}}", hook, from=1-1, to=1-2]
\end{tikzcd}\]
A very natural question to ask is if we can extend the inclusion of \(\heart\) into \(\mc T \) into a triangulated functor on \(D(\heart)\) (called the realization functor) \[\begin{tikzcd}
	{\heart} & {\mathcal{T}} \\
	{D(\heart)}
	\arrow["{i_{D(\heart)}}"', hook, from=1-1, to=2-1]
	\arrow["{i_\mathcal{T}}", hook, from=1-1, to=1-2]
	\arrow["{\exists real}"', dotted, from=2-1, to=1-2]
\end{tikzcd}\]
in such a way that \(real(i_{D(\heart)}) \simeq i_{\mc T}\)? For the bounded derived category, an affirmative answer is reached provided there is a so-called `f-category' living above \(\mc T\) (see \cite[Appendix]{realizationforfcategories}). Such a condition is not too strict in practice, as Modoi \cite{filteredenhance} proved that every triangulated category which is the underlying category of a stable derivator admits such an f-category. In particular, as every stable model category gives rise to a stable derivator (see for example \cite[Example~1.11]{moritatheoryderiv}), any triangulated category that arises as the homotopy category of a stable model category admits an f-category. 
\end{remark} 
\begin{remark}
    In \cite{moritatheoryderiv}, Virili is able to construct an unbounded realization functor for a t-structure with some assumptions about some higher structures living above \(\mc T\). Let us briefly set the stage for the following theorem. Suppose that \(\mc T\) arises as the base of a strong, stable derivator \(\mathbb{D}\), that is \(\mc T = \mathbb D(e)\), and let \(\textbf{t}\) be a t-structure on \(\mc T\). Then Virili is able to construct a morphism of derivators \[real_{\textbf{t}}: Ch_{\heart} \to \mathbb D\]
    where \(Ch_{\heart}\) is the derivator defined by \(Ch_{\heart}(I) = Ch((\heart)^I)\) for any diagram category \(I\). Moreover, this morphism of derivators takes quasi-isomorphisms to isomorphisms (\cite[Theorem~6.7]{moritatheoryderiv}), so, assuming the derived category \(D(\heart)\) exists, the morphism factors through the derivator underlying the derived category to give a morphism of derivators \[real_{\textbf{t}}: D_{\heart} \to \mathbb D\]
    Moreover, assuming the t-structure satisfies certain properties we can say more about this morphism. We summarize all this below, rephrased in the language of triangulated categories.
\end{remark} 
\begin{theorem}\cite[\S6.3]{moritatheoryderiv}\label{realizationfunct}
Suppose \(\mc T\) arises as the base of a strong and stable derivator \(\mathbb{D}\) (for example, \(\mc T\) can be the homotopy category of a stable model category), and let \(\textbf{t}\) be a t-structure on \(\mc T\). Denote by \[real^b : D^b(\mc A) \to \mc T\] the bounded realization functor (which exists by \cite{filteredenhance}). Then we can lift the bounded realization functor to an unbounded one \[real: D(\mc A) \to \mc T.\]
Moreover, suppose that \begin{enumerate}
    \item [(a)]\(\textbf{t}\) is non-degenerate,
    \item [(b)] \(real^b\) is fully faithful, and
    \item [(c)] \(\textbf{t}\) is smashing.
\end{enumerate} 
Then the unbounded realization functor remains fully faithful and is coproduct preserving. 
\end{theorem}
\begin{remark}
    We will now introduce the t-structure that exists for \(DTM(\olQ, \QQ)\). This was first studied by Levine in \cite{tatemotiveststructure}, and then expanded upon by Wildeshaus in \cite{wildehaus}. They state their results at the level of compact objects, however, as we will see their constructions generalize to the unbounded categories. For ease of notation let us write \(\mc T: = DTM(\olQ, \QQ)\) for what follows.
\end{remark}
\begin{definition}
Define \(\mc T_{[a,b]} := loc(\QQ(n) : a \le -2n \le b )\). We allow \(a, b \in \{-\infty, \infty\}\), in which case we just have \(\mc T = \mc T_{(-\infty, \infty)}\). We write \(\mc T_a := \mc T_{[a,a]}\).
\end{definition}
\begin{remark}
     It follows that \(\mc T_a \simeq D(\QQ) \) (see \cite[\S1]{tatemotiveststructure}).
\end{remark}
We can define a t-structure on \(\mc T_{[a,b]}\) as follows: 
\begin{lemma}
Let \(a \le b \le c\). Then \(\textbf{t} := (\mc T_{[a,b-1]}, \mc T_{[b,c]})\) is a t-structure on \(\mc T_{[a,c]}\).
\end{lemma}
\begin{proof}
This is \cite[Lemma~1.2]{tatemotiveststructure}.
\end{proof}
\begin{remark}\label{gr_acoprodpreserving}
    Let \(a \in \ZZ\). The cohomology functor associated to the t-structure \((\mc T_{[\infty, a-1]}, \mc T_{[a, \infty]})\) is denoted \(gr_a : = H^0_a\). Note that this t-structure is smashing by construction, so \(gr_a\) is coproduct preserving by Remark \ref{smashing t struct}.
\end{remark} 
\begin{definition}\cite[Definition~1.4]{tatemotiveststructure}
Let \(a\) be even. Define \(\mc T^{\ge 0}_a\) to be the full, additive category generated by \(\Sigma^n\QQ(\frac{a}{2})\) for \(n \le 0\). Similarly define \(\mc T^{\le 0}_a\) to be the full, additive category generated by \(\Sigma^n\QQ(\frac{a}{2})\) for \(n \ge 0\). Let \(\mc T^{\ge 0}\) be the full subcategory of \(\mc T\) of objects \(X \in \mc T\) such that \(gr_c(X) \in \mc T^{\ge 0}_c\) for all \(c \in \ZZ\). Similarly, define \(\mc T^{\le 0}\) be the full subcategory of \(\mc T\) of objects \(X \in \mc T\) such that \(gr_c(X) \in \mc T^{\le 0}_c\) for all \(c \in \ZZ\).
\end{definition}
\begin{remark}
Under the equivalence \(D(\QQ) \simeq \mc T_a\) we have that the pair \((\mc T^{\le 0}_a, \mc T^{\ge 0}_a )\) on \(\mc T_a\) corresponds to the standard t-structure on \(D(\QQ)\). 
\end{remark}
\begin{remark}\cite[\S4]{tatemotiveststructure}
To spell it out a little more, we have that:\begin{align*}
      X \in \mc T^{\le 0} \iff gr_a(X) \simeq \coprod_{n \ge 0} \Sigma^n \QQ(\frac{a}{2})^{m_n} \text{ for all } a \in \ZZ\\
      X \in \mc T^{\ge 0} \iff gr_a(X) \simeq \coprod_{n \le 0} \Sigma^n \QQ(\frac{a}{2})^{m_n} \text{ for all } a \in \ZZ
\end{align*}
\end{remark}
\begin{theorem}\label{big t structure theorem}\cite[Theorem~4.2]{tatemotiveststructure}, \cite[Lemma~3.3]{peterspecQ}
The pair \((\mc T^{\le 0}, \mc T^{\ge 0})\) is a t-structure on \(\mc T\). Moreover, \begin{enumerate}
    \item [(a)] The t-structure is non-degenerate, compatible and strongly hereditary.
    \item [(b)] The heart \(\AT\) is generated as a full abelian category closed under extensions and coproducts by \(\QQ(n), n \in \ZZ\).
    \item [(c)] Every object \(X \in \AT\) admits a functorial filtration by sub-objects \[ 0 \subset \dots \subset Z_n(X) \subset Z_{n+1}(X) \subset \dots \subset  X\]
    whose sub-quotients \(Z_{n}(X)/Z_{n-1}(X)\) are coproducts of shifts of \(\QQ(n)\).
    \item [(d)] The t-structure is smashing.
\end{enumerate}
\end{theorem}
\begin{proof}
Parts (a)-(c) are \cite[Theorem~4.2]{tatemotiveststructure} and \cite[Lemma~3.3]{peterspecQ}. The only change is in Peter's lemma, where our t-structure will no longer be bounded, but that is unimportant for us. Moreover, part (d) follows from our definition of \(\mc T_a\), the description of the aisle and co-aisle in the preceding remark, and that \(gr_a\) commutes with coproducts (recall Remark \ref{gr_acoprodpreserving}).
\end{proof}
\begin{corollary}
There is an equivalence of triangulated categories \[real: D(\AT) \simeq DTM(\olQ, \QQ).\]
\end{corollary}
\begin{proof}
We can apply Theorem~\ref{realizationfunct} to get an exact and coproduct preserving functor \(real: D(\AT) \to DTM(\olQ, \QQ)\). To prove that \(real^b\) is fully faithful, since \(\AT\) generates \(D^b(\AT)\) as a triangulated category, it suffices to prove that \begin{align*}
    Ext^i(M, N) \xrightarrow{real^b} & hom_{\dtmq}(real^b(M), real^b(N)[i])\\
     =&\hspace{.05cm} hom_{\dtmq}(M,N[i])
\end{align*} is a bijection for objects \(M, N \in \AT\). Since the t-structure is strongly-hereditary, this follows from Lemma \ref{ext for hereditary cat}. Hence \(real^b\) is fully faithful, which implies \(real\) is fully faithful by Theorem \ref{realizationfunct}. Thus the essential image is a localizing subcategory of \(DTM(\olQ, \QQ)\). This category contains all the Tate twists by Theorem \ref{big t structure theorem}(b). Since the Tate twists generate \(DTM(\olQ, \QQ)\) we get that the essential image is everything.
\end{proof}
This equivalence immediately gives us the following:
\begin{corollary}
    Every object \(t \in DTM(\olQ, \QQ)\) is isomorphic to \(t \simeq \coprod_i \Sigma^{-i}H^i(t)\)
\end{corollary}
\begin{proof}
    This isomorphism holds in \(D(\AT)\) because \(\AT\) is heriditary, see for example \cite{everyobjectcoproductofcohomologyD(A)}. The result thus follows for \(\dtmq\) since \(real\) is a coproduct preserving equivalence, and is just the identity on the heart.
\end{proof}
We now turn to classifying the localizing ideals of \(DTM(\olQ, \QQ)\). We will closely follow Peter's approach of computing the thick tensor ideals of \(\dtmq^c\).  
\begin{remark}
Let us fix some notation for the rest of this section. We have the additive functor \begin{align*}
    \phi :\hspace{.1cm} & DTM(\olQ, \QQ) \to \AT \\
     &X  \longrightarrow  \coprod_{i} H^i(X)
\end{align*}
and the inclusion functor \[\iota: \AT \hookrightarrow DTM(\olQ, \QQ).\]
Note that the functor \(\phi\) is also a tensor-functor, because the t-structure is compatible; see \cite{kunnethform}. Following the terminology of Peter we say \(\mc M \subset \AT\) is a coherent tensor ideal of \(\AT\) if it is closed under extensions, kernels, cokernels, and tensoring by arbitrary elements of \(\AT\). Let us denote by \(Coh^{\coprod}(\AT)\) to be the set of coherent tensor ideals of \(\AT\) closed under coproducts.
\end{remark}
\begin{proposition}\label{bijectionoflocidealsDTMAT}
The maps 
    \[\begin{tikzcd}
	\{\text{Localizing ideals of } DTM(\olQ, \QQ) \} & Coh^{\coprod}(\AT)
	\arrow["\iota^*", shift left=1, from=1-1, to=1-2]
	\arrow["\phi ^*", shift left=1, from=1-2, to=1-1]
\end{tikzcd}\]
defined by \begin{align*}
    \iota^*(\mc L) & = \mc L \cap \AT = \iota^{-1}(\mc L)\\
    \phi^*(\mc M) & = \phi^{-1}(\mc M)
\end{align*} are bijective functions inverse to each other.
\end{proposition}
\begin{proof}
This is the unbounded analogue of \cite[Theorems~4.4,~4.11]{peterspecQ}. We include a sketch of the argument here for the readers convenience. The first step is to prove \(\phi^*\) is well defined, so let \(\mc M \subset \AT\) be a localizing ideal of \(\AT\). The same proof as in \cite[Theorem~4.4]{peterspecQ} shows that \(\phi^*(\mc M)\) is closed under extensions and tensors. It is closed under coproducts because the cohomology functors preserve coproducts since the t-structure is smashing, and because \(\mc M\) is closed under coproducts by assumption. Now we observe that since \(\phi \circ \iota \simeq id_{\AT}\) we have \(\iota^*(\phi^*(\mc M)) = (\phi \circ \iota )^{-1}(M) = M\), so we get that \(\iota^*\) is well defined on the image of \(\phi^*\), and is left inverse to it. Thus, it remains to show \(\phi^*\) is surjective. To do so, we can apply \cite[Theorem~4.11]{peterspecQ} after noting that we still have an isomorphism  \[t \simeq \coprod_{i \in \ZZ} \Sigma^{-i}H^{i}(t)\]
for any \(t \in DTM(\olQ, \QQ)\).
\end{proof}
\begin{remark}
    We can now complete our last reduction step: we continue to denote \(\mc T = \dtmq\). For each \(a \in \ZZ\) recall we have the functor \(gr_a: \mc T \to \mc T_a \simeq D(\QQ) \simeq Gr_{\ZZ}(Vect_{\QQ})\). Moreover, Biglari showed (\cite[Proposition~3.8]{onfinitedimoftatemotives}) the following: 
    \begin{proposition}
        The functor \begin{align*}
            gr : = \coprod_{a \in \ZZ}gr_a [-a] : \mc T \to D(\QQ)\\
            X \mapsto \coprod_{a \in \ZZ} gr_a(X)[-a]
        \end{align*} 
        is a tensor-triangulated functor.
    \end{proposition}   
    If we restrict the domain of this functor to just the heart, we obtain the following \cite[Corollary~4.3]{tatemotiveststructure}: 
    \begin{proposition}
        The functor \(gr |_{\AT}: \AT \to D(\QQ)\) gives an equivalence of \(\AT\) with a tensor subcategory of \(Gr_{\ZZ}(Vect_{\QQ})\).
    \end{proposition}
\end{remark} 
With these two propositions in hand, we can run the exact argument as in \cite[Theorem~4.15]{peterspecQ} to conclude:
\begin{corollary}\label{locidealsofAT}
The only coherent tensor ideals closed under coproducts of \(\AT\) are the \((0)\) ideal and \(\AT\).
\end{corollary}
\begin{theorem}
The category \(DTM(\olQ, \QQ)\) is stratified.
\end{theorem}
\begin{proof}
Since \(Spc(\dtmq^c) = \{*\}\) is just a point, we have that \(g(*) = \unit \ne 0\), and the local-to-global principle holds. Moreover, by combining Proposition~\ref{bijectionoflocidealsDTMAT} and Corollary~\ref{locidealsofAT}, we get that there are only two localizing ideals of \(\dtmq\). Hence the localizing ideal \(g(*)\) generates has to be minimal. 
\end{proof}
Again keeping in mind Remark~\ref{height0reduction}, this gives us:
\begin{corollary}\label{minatmo}
    Minimality holds at the height zero prime \(m_0\).
\end{corollary} 
\section{Minimality at the Height One Prime}\label{e_psection}
In this section we establish minimality for the height one prime \(e_p\). 
\begin{remark}\label{reduction steps for e_p}
From Proposition \ref{localresult} establishing minimality at the prime \(e_p\) is equivalent to establishing minimality at the closed point in \[\dtmp / \lt(e_p).\]
     Combining Theorems \ref{locgivesetal}, \ref{specdtmfinite} and \ref{maintheoremmartin}, we get that \[\dtmp / \lt(e_p) \simeq \dtmep\]
and so we are reduced to establishing minimality at the closed point in the category \(\dtmep\). To do so, we will consider the geometric functor \begin{align} \label{geofunctor}
    \gamma^*: \dtmep \to \DTMe(\olQ, \ZZ/p)
\end{align} 
induced by the ring map \(\ZZ_{(p)} \to \ZZ/p\). Note that \(\ZZ_{(p)}\) is a regular ring and so we have \(\ZZ/p \in D(\ZZ_{(p)})^c\). Hence \(\gamma_*(\unit_{\DTMe(\olQ, \ZZ/p)}) \in \dtmep^c \) by Remark~\ref{if R' in D(R)^c then strongly closed}.
\end{remark}
This category \(\DTMe(\olQ, \ZZ/p)\) is quite simple in fact.
\begin{theorem}[Rigidity Theorem]\label{rigiditytheorem}
There is a tt-equivalence \[\DMe(\olQ, \ZZ/p) \simeq D(\ZZ/p)\]
\end{theorem}
\begin{proof}
    This is \cite[Theorem~4.5.2]{etalemotives}.
\end{proof}
\begin{corollary}\label{subfields are fields}
   The category \(\dtmep\) is stratified.
\end{corollary} 
\begin{proof}
     As the derived category of a field is, of course, stratified, we obtain the same for \(\DMe(\olQ, \ZZ/p)\). We can then apply \cite[Proposition~5.22]{ttfieldsrumination} to conclude that \(\DTMe(\olQ, \ZZ/p)\) is itself a tt-field, and thus is also stratified (see \cite[Theorem~18.4]{cosupport}). 
\end{proof}
\begin{remark}
    On spectra the geometric functor (\ref{geofunctor}) gives
\[\begin{tikzcd}
	{Spc(\DTMe(\olQ, \ZZ/p)^c)} & {Spc(\dtmep^c)} \\
	{(0)} & {e_p} \\
	& {m_0}
	\arrow[no head, from=2-2, to=3-2]
	\arrow["{Spc(\gamma_*)}", hook, from=1-1, to=1-2]
	\arrow["{Spc(\gamma_*)}", shift right=5, hook, from=2-1, to=2-2]
\end{tikzcd}\]
We can therefore apply Theorem \ref{compacttheorem} and we immediately get:
\end{remark}
\begin{corollary}\label{minatep}
Minimality holds at the height 1 primes \(e_p\).
\end{corollary}
\begin{proof}
    By Remark~\ref{reduction steps for e_p} we need to prove the category \(\dtmep\) satisfies minimality at the unique closed point. Moreover we can apply Proposition~\ref{compacttheorem}, since ~\(\gamma_*(\unit_{\DTMe(\olQ, \ZZ/p)}) \in \dtmep^c \), to focus on establishing minimality at the unique prime in the local category \(\DTMe(\olQ, \ZZ/p)\). This then follows from Corollary \ref{subfields are fields}.
\end{proof}
\begin{theorem}
    The derived category of étale motives, \(\DTMe(\olQ, \ZZ)\) is stratified.
\end{theorem}
\begin{proof}
    As \(Spc(\DTMe(\olQ, \ZZ)^c)\) is noetherian it automatically satisfies the local-to-global principle. Hence we just need to show minimality holds at each prime. By Theorem \ref{maintheoremmartin}, these primes are \(m_0\) and \(e_p\) for \(p\) prime, and we showed minimality holds at these primes in Corollaries \ref{minatep} and \ref{minatmo}.
\end{proof}
\section{Brown--Adams Representability and Stratification}\label{brownadamssection}
In this section, we return to Question~\ref{questioninintro} from the introduction in order to prove minimality holds at the height two prime \(m_p\) in Chapter \ref{m_psection}. To remind the reader, we suppose we are given two rigidly-compactly generated tt-categories \(\mc T_1\) and \(\mc T_2\), where \(\mc T_1\) is stratified and that we have a tt-equivalence between their compact parts \(\mc T_1^c \simeq \mc T_2^c\). Our question is whether this then implies that \(\mc T_2\) is stratified as well. While we are not able to answer this question in full generality, we are able to achieve positive results assuming our category satisfies Brown--Adams Representability, which we now explain. 

Let \(\mc T\) be a rigidly-compactly generated tt-category. The \emph{category of modules} on \(\mc T\) is the Grothendieck abelian category \[\mc A: = Mod(\mc T^c) : = Add((\mc T^c)^{op}, Ab)\] 
of contravarient additive functors from \(\mc T^c\) to abelian groups. The subcategory of finitely presented modules \(\mc A ^{fp} := mod(\mc T^c)\) coincides with the usual Freyd-envelope of \(\mc T^c\) \cite[Chapter~5]{neemantriangulatedcatbook}. For relevant information about this category of modules, we will be primarily following the recent series of papers by Balmer, Krause and Stevenson \cite{homsupport, ttfieldframesmashing, ttfieldnilpotence, ttfieldsrumination}.
\begin{remark}
    We have the restricted Yoneda functor \begin{align*}
    h: \hspace{.1cm}&\mc T \xrightarrow{}  \mc A \\
    & t \xrightarrow{} \hat{t}: = Hom(-, t)|_{\mc T^c}
\end{align*} which fits into the commutative square \[\begin{tikzcd}
	\mc T^c & \mc A^{fp} \\
	\mc T & \mc A.
	\arrow["h", hook, from=1-1, to=1-2]
	\arrow[hook, from=1-1, to=2-1]
	\arrow[hook, from=1-2, to=2-2]
	\arrow["h", from=2-1, to=2-2]
\end{tikzcd}\]
Note that restricted Yoneda is no longer an embedding in general, due to the potential existence of so-called phantom maps. Nevertheless, \(h: \mc T \to \mc A\) is conservative, that is \(\hat{t} = 0 \implies t=0\), because \(\mc T\) is compactly generated.  
\end{remark} 
\begin{definition} \label{dloc def}
     For \(S \subseteq \mc A\) we write \begin{enumerate}
     \item \(Loc_{\mc A}(S)\) to be the smallest Serre subcategory containing \(S\) closed under coproducts and suspension.
     \item \(Locid_{\mc A}(S)\) the smallest Serre subcategory containing \(S\) closed under coproducts and tensor products (and hence is automatically closed under suspension, see \cite[Remark~2.2]{ttfieldsrumination}).
 \end{enumerate}
\end{definition}
\noindent We summarize the facts we will need about \(\mc A\) below. This can be found in \cite[\S\S2-3]{ttfieldsrumination}. \begin{enumerate}
    \item [(1)] \(\mc A\) inherits a suspension \(\Sigma_{\mc A}\) such that \(h \circ \Sigma_{\mc T} = \Sigma_{\mc A} \circ h\).
    \item [(2)] \(\mc A\) is closed symmetric monoidal under Day convolution. Under this tensor product \(h: \mc T \to \mc A\) is symmetric monoidal, and moreover \(\hat{t} \otimes -\) is exact and colimit preserving for any \(t \in \mc T\). 
\end{enumerate}
\begin{definition}\label{homfunctordef}
     Sitting inside of \(\mc A\) is the full subcategory of \emph{Homological Functors}, which we will denote \(Hol(\mc T^c)\), consisting of those contravariant functors that send exact triangles to long exact sequences.  Clearly, \(\hat{t} \in Hol(\mc T^c)\) for any \(t \in \mc T\). Moreover, in certain important historic examples there has been a much stronger relationship between the essential image of restricted Yoneda and the subcategory of homological functors: 
\end{definition} 
\begin{theorem}\cite{OGbrownadamrepinSH}
    Let \(\mc T = \text{SH}\) denote the stable homotopy category, and \(\mc T^c\) the category of finite spectra. Then any homological functor \[\mc H: (\mc T^c)^{\text{op}} \to \mc Ab\]
    is isomorphic to \(h(t)\) for some \(t \in \mc T\). Moreover, any natural transformation \[h(t) \to h(s)\] is induced by some (non-unique) map \[t \to s.\]
\end{theorem}
\begin{remark}\label{brownadamsdef}
    Neeman, Keller and Christensen have investigated the extent to which this result generalizes to other rigidly-compactly generated tt-categories in \cite{onatheoremofbrownandadams, failureofbrownrepingeneral}. Following their terminology, we say that \(\mc T\) satisfies: \begin{enumerate}
    \item [(BRO)] If every \(\mc H \in Hol(\mc T^c)\) is isomorphic to \(h(t)\) for some \(t \in \mc T\).
    \item [(BRM)] If every natural transformation \(h(t) \to h(s)\) is induced by some (potentially non-unique) morphism \(t \to s\).
\end{enumerate}  
\end{remark}  
\begin{remark}
    It follows from results of Beligiannis (see \cite[Theorem~11.8]{relativehomologicalalg}) that (BRM) implies (BRO).
\end{remark}
\begin{definition}
    We say that \(\mc T\) satisfies \emph{Brown--Adams representability} if condition (BRM) holds.
\end{definition}
\begin{remark}\label{countable imples brown adams}
    Neeman shows \cite[Proposition~4.11,~Theorem.~5.1]{onatheoremofbrownandadams} that not every triangulated category satisfies Brown--Adams representability. However, he established a sufficient condition, namely that \(\mc T\) satisfies Brown--Adams representability if \(\mc T^c\) is equivalent to a countable category (that is, a category with only countably many objects and morphisms between them).
\end{remark}
\begin{remark}\label{showing iso using brown adams}
    Suppose \(\mc T\) is a compactly generated tt-category that satisfies Brown--Adams representability. Then it follows (see \cite[Remark~3.2]{onatheoremofbrownandadams}) that any isomorphism \(\hat{s} \xrightarrow{\sim} \hat{t}\) in \(\mc A\) is induced by an isomorphism \(s \xrightarrow{\sim} t\) in \(\mc T\).
\end{remark}
Let us now connect this back to the theory of stratification: 
\begin{definition}\cite[Remark~3.4]{ttfieldnilpotence}\label{homspectrumdef}
 The \emph{homological spectrum} of \(\mc T^c\), denoted \(Spc^h(\mc T^c)\), is the set of all maximal Serre tensor ideals in \(\mc A^{fp}\). We will refer to \(\beta \in Spc^h(\mc T^c)\) as a \emph{homological prime}.
\end{definition}
\begin{remark} \cite[Proposition~2.4]{homsupport}
     Let \(\beta \in Spc^h(\mc T^c)\) be a homological prime, and consider the quotient \[\overline{h}: \mc T \to \mc A \to \mc A /  Locid_{\mc A}(\beta). \]
     There is a unique pure-injective object \(E_{\beta} \in \mc T\) such that \[Locid_{\mc A}(\beta) = Ker( \hat{E}_{\beta} \otimes - ).\]
 \end{remark}
\begin{remark}\label{fieldEbobjects}
     These pure-injective objects \(E_{\beta}\), while abstractly defined, are often nice, recognizable objects. For example, in the derived category of a ring, we have that \(Spc^h(D(R)^c) \simeq Spc((D(R)^c)\), and recalling that \(Spc(D(R)^c) \simeq Spec(R)\), the \(E_{\beta}\) object corresponding to \(p \in Spec(R)\) is isomorphic to the residue field \(\kappa(p)\). Similarly in SH, we have \(Spc^h(\text{SH}^c) \simeq Spc(\text{SH}^c)\) and under this correspondence the \(E_\beta\) objects are isomorphic to the Morava K-theories \(K(p,n)\) (see \cite[Corollaries~3.3,~3.6]{balmercameron}). 
\end{remark}
\begin{definition}\label{field objects definition}
    We say an object \(F \in \mc T\) is a \emph{field object} if, for any \( t \in \mc T\), we have that \(t \otimes F \) is a coproduct of suspensions of \(F\).
\end{definition}
\begin{remark}
    In both examples in Remark \ref{fieldEbobjects}, the \(E_{\beta}\) objects are field objects. 
\end{remark}
\begin{remark}
Let us assume that \(\mc T\) is stratified. Then it follows that \(\mc T = \lt (E_\beta: \beta \in Spc^h(\mc T^c))\). Indeed, the local-to-global principle for \(\mc T\) tells us we have \(\mc T= \lt (g(\mc P): \mc P \in Spc(\mc T^c) )\). Moreover, since \(\mc T\) is stratified, the canonical comparison map \(\phi: Spc^h(\mc T^c) \to Spc(\mc T^c)\) is a homeomorphism \cite[Thrm.~4.7]{comparisonsupports}, so let us denote \(\beta_\mc P\) to be the unique homological prime corresponding to \(\mc P \in Spc(\mc T^c)\). It follows \cite[Lemma.~3.7]{comparisonsupports} that \(E_{\beta_\mc P} \in \lt (g(\mc P))\), and since \(\mc T\) is stratified, we get an equality \(\lt(E_{\beta_\mc P}) = \lt ( g(\mc P))\).
\end{remark}
\begin{remark}\label{assumptionsforthissection}
Let us now show how this can help us pass stratification from one category to another. Let \(\mc T_1\) and \(\mc T_2\) be rigidly-compactly generated tt-categories  and let \(F: \mc T_1^c \xrightarrow{\sim} \mc T_2^c\) be a tt-equivalence. Then this equivalence induces: \begin{enumerate}
    \item [(1)] an exact, tensor equivalence \(\hat{F}: \mc A_1 \xrightarrow{\sim} \mc A_2\); and
    \item [(2)] a homeomorphism \( Spc^h(F): Spc^h(\mc T_2^c) \xrightarrow{\sim} Spc^h(\mc T_1^c)\).
\end{enumerate}
Let \(\beta_1 \in Spc^h(\mc T_1^c)\) and let \(\beta_2\) be the unique homological prime in \(Spc^h(\mc T_2^c)\) mapping to \(\beta_1\). Consider the following diagram: \[\begin{tikzcd}
	\mc A_1 && \mc A_2 \\
	\\
	\overline{\mc A_1 } := \mc A_1 / Locid_{\mc A}(\beta_1) && \mc A_2 / Locid_{\mc A}(\beta_2) := \overline{\mc A_2}.
	\arrow["\widehat{F}", shift left=1, from=1-1, to=1-3]
	\arrow["{\widehat{L}_1}"', shift right=2, two heads, from=1-1, to=3-1]
	\arrow["{\widehat{R}_1}"', shift right=2, hook, from=3-1, to=1-1]
	\arrow["{\widehat{L}_2}"', shift right=2, two heads,  from=1-3, to=3-3]
	\arrow["{\widehat{R}_2}"', shift right=1, hook, from=3-3, to=1-3]
	\arrow["{\overline{F}}"', shift left=2, from=3-1, to=3-3]
\end{tikzcd}\] 
\end{remark}
\begin{proposition}
Keeping the notation and set up as in the above remark \ref{assumptionsforthissection}, we have
    \begin{enumerate}
    \item [(a)] \(\overline{F} \circ \widehat{L}_1 \simeq \widehat{L}_2 \circ \widehat{F}.\)
    \item [(b)] \(\overline{F}: \overline{\mc A_1 } \to  \overline{\mc A_2}\) is an equivalence.
    \item [(c)] \(\widehat{F} \circ \widehat{R}_1 \simeq \widehat{R}_2 \circ \overline{F}.\)
\end{enumerate}
\end{proposition}
\begin{proof}
    Parts (a) and (b) follow by definition of the \(\beta_i's\) and the fact that giving a Serre subcategory of \(\mc A^{fp}\) is equivalent to giving the localizing category it generates (see for example, \cite[Appendix~A. Remark~8]{ttfieldframesmashing}). Let us now prove (c). First note that (a) and (b) imply that we also have  \(\overline{F}^{-1} \circ \hat{L}_2 \simeq \hat{L}_1 \circ \hat{F}^{-1} \). Then for an arbitrary object \(c \in \overline{\mc A}_1\) we compute \begin{align*}
        \mc A_2 ( - , \hat{R}_2 \overline{F}c) & \simeq \overline{\mc A}_2 ( \hat{L}_2 ( - ) , \overline{F}(c) )\\
        & \simeq \overline{\mc A}_1 ( \overline{F}^{-1}\hat{L}_2 ( - ) , c) \\
        & \simeq \overline{\mc A}_1 ( \hat{L}_1 \hat{F}^{-1}( - ), c) \\
        & \simeq \mc A_1 ( \hat{F}^{-1}( - ), \hat{R}_1 (c) )\\
        & \simeq \mc A_2 ( - , \hat{F} \hat{R}_1 (c))
    \end{align*}
    and then we summon Yoneda.
\end{proof}
\begin{remark}
  The objects \(E_{\beta_i}\) are uniquely determined by the injective hull of the unit in \(\overline{\mc A_i}\). That is, letting \(\overline{\unit} \to \overline{E}_{\beta_i}\) be this injective hull, we have that \(E_\beta\) is the unique object in \(\mc T\) such that \(\widehat{E}_{\beta_i} = \hat{R}_i(\overline{E}_{\beta_i}) \in \mc A_i\) (see, for example \cite[2.11]{homsupport}).
\end{remark} 
\begin{proposition}\label{ebsgotoebs}
    Let \(\mc T_i, \mc A_i,  \beta_i\) be as in Remark \ref{assumptionsforthissection}. Then \(\hat{F}(\hat{E}_{\beta_1}) = \hat{E}_{\beta_2}\).
\end{proposition}
\begin{proof}
    We have that \(\overline{F}\) will send the injective hull of the unit in \(\overline{\mc A_1}\) to the injective hull of the unit in \(\overline{\mc A_2}\). Hence we know that \(\overline{F}(\overline{E}_{\beta_1}) \simeq \overline{E}_{\beta_2}\). Then we compute that \begin{align*}
        \widehat{E}_{\beta_2} & = \widehat{R}_2(\overline{E}_{\beta_2})\\
         & = \widehat{R}_2 ( \overline{F}(\overline{E}_{\beta_1})) \\
         & = \widehat{F} (\widehat{R}_1( \overline{E}_{\beta_1}) ) \\
         & = \widehat{F} (\widehat{E}_{\beta_1})
    \end{align*}
    which is precisely what we wanted.
\end{proof}
\begin{hypothesis}\label{hypothesis for brown adams}
Let \(\mc T_1\) and \(\mc T_2\) be rigidly-compactly generated tt-categories. Suppose \begin{enumerate}
    \item [(1)] \(\mc T_1\) is stratified;
    \item [(2)]\(\mc T_1^c\) is equivalent to a countable category;
    \item [(3)] \(F: \mc T_1^c \to \mc T_2^c\) is a tt-equivalence; and 
    \item [(4)] For every nonzero homological functor \(\hat{t} \in \mc A_1\) there exists a nonzero map \(\hat{E}_\beta \otimes \hat{x} \to \hat{t}\) for some \(\beta \in Spc^h(\mc T_1^c)\) and compact \(x \in \mc T_1^c\).
\end{enumerate} 
We shall also fix \(\beta_1 \in Spc^h(\mc T_1^c)\) and let \(\beta_2\) denote the unique homological prime in \(Spc^h(\mc T_2^c)\) mapping to \(\beta_1\) as in Remark \ref{assumptionsforthissection}.
\end{hypothesis}
\begin{remark}
    Let \(\mc T_1\) and \(\mc T_2\) be two rigidly-compactly generated tt-categories and suppose they satisfy conditions \((1) - (3)\) of Hypothesis \ref{hypothesis for brown adams}. Then it follows from Proposition \ref{ebsgotoebs} that \(\mc A_1\) satisfies condition \((4)\) if and only if \(\mc A_2\) satisfies it.
\end{remark}
\begin{proposition}\label{E_b generate cat}
   Keeping everything as in Hypothesis \ref{hypothesis for brown adams}, we have \[\mc T_2 = Locid(E_\beta: \beta \in Spc^h(\mc T_2^c))\]
\end{proposition}
\begin{proof}
    First note that \(\mc T_2\) also satisfies Brown--Adams representability since \(\mc T_2^c\) is also countable due to the equivalence \(F: \mc T_1^c \to \mc T_2^c\). Hence restricted Yoneda is full so we have that for any nonzero \(t \in \mc T_2\) there is a nonzero map \( E_\beta \otimes x \to t\) for some compact \(x \in \mc T_2^c\), and \(\beta \in Spc^h(\mc T_2^c)\). Denoting \([-,-]\) to be the internal hom in \(\mc T_2\), we have \[Hom_{\mc T_2}( E_\beta \otimes x, t) \simeq Hom_{\mc T_2}( x , [E_\beta, t]) \]
    which implies that \([E_\beta, t]\) must be nonzero. Hence we have the following string of implications (compare the following proof to \cite[Theorem~6.4]{cosupport}): \begin{align*}
        t = 0 & \iff [E_\beta, t] = 0 \text{ for all } \beta \in Spc^h(\mc T_2^c) \\
        & \iff t \in \{E_\beta:  \beta \in Spc^h(\mc T_2^c)\}^\perp\\
        & \iff t \in (Locid(E_\beta: \beta \in Spc^h(\mc T_2^c)))^\perp.
    \end{align*}
In other words, letting \(\mc L := Locid(E_\beta: \beta \in Spc^h(\mc T_2^c)) \), we have that \(\mc L ^\perp = 0 \). Now, because \(\mc L\) is set-generated, it is a strictly localizing tensor ideal, see for example \cite[Proposition~3.5]{BHSStrat}. As a consequence we have that \(\mc L ={}^\perp(\mc L^\perp) \); see \cite[Remark~2.11]{cosupport}. Thus we have \(\mc L = {}^\perp(\mc L^\perp) = {}^\perp 0 = \mc T_2\). 
\end{proof}
\begin{proposition}\label{fieldobjects}
    Keep the assumptions as in Hypothesis \ref{hypothesis for brown adams} and suppose further that the \(E_\beta\) objects in \(\mc T_1\) are field objects, as in Definition \ref{field objects definition}. Then the same is true for the \(E_\beta\) objects in \(\mc T_2\). 
\end{proposition}
\begin{proof}
    Note that, because of Brown--Adams representability, it suffices to prove that \[\hat{E}_{\beta_2} \otimes \hat{t} \simeq \coprod_{j \in I} \Sigma^{m_j} \hat{E}_{\beta_2}\]
    in \(\mc A_2\) (recall Remark \ref{showing iso using brown adams}). Moreover, since \(\hat{F}\) is an equivalence, we can check this after applying \(\hat{F}^{-1}\). Take \( 0 \ne t \in \mc T_2\). Then we have that \begin{align*}
        \widehat{F}^{-1}(\hat{t} \otimes \widehat{E}_{\beta_2}) & \simeq \hat{F}^{-1}(\hat{t}) \otimes \hat{F}^{-1}(\widehat{E}_{\beta_2}) \\
        & \simeq \hat{F}^{-1}(\hat{t}) \otimes \widehat{E}_{\beta_1}\\
        & \simeq \coprod_{j \in I} \Sigma^{m_j} \widehat{E}_{\beta_1} \\
        & \simeq \coprod_{j \in I} \Sigma^{m_j} \hat{F}^{-1}(\widehat{E}_{\beta_2}) \\
        & \simeq \hat{F}^{-1} \big ( \coprod_{j \in I} \Sigma^{m_j} (\widehat{E}_{\beta_2}) \big)
    \end{align*} Hence we have that \(\hat{t} \otimes \hat{E}_{\beta_2} \simeq \coprod_{j} \Sigma^{m_j} \hat{E}_{\beta_2}\). Since Brown--Adams representability holds for \(\mc T_2\) this isomorphism in \(\mc A_2\) is witnessed by an isomorphism in \(\mc T_2\).
\end{proof}
\begin{remark}
    The above assumption that the \(E_\beta\) objects are field objects is not as strong as it may sound. Indeed, recall from Remark \ref{fieldEbobjects} that they are field objects in many examples of interest.
\end{remark}
\begin{theorem} \label{stratifiedusingbrownadamstheorem}
    Let \(\mc T_1\) and \(\mc T_2\) be as in Hypothesis \ref{hypothesis for brown adams}. If the \(E_\beta\) objects are field objects (in either \(\mc T_1\) or \(\mc T_2\)), then \(\mc T_2\) is stratified.
\end{theorem}
\begin{proof}
     To prove that \(\mc T_2\) is stratified, we need to prove that the local-to-global principle holds and that \(Locid(g(\mc P) )\) is a minimal localizing ideal for all primes \(\mc P \in Spc(\mc T_2^c)\). Now, since \(\mc T_1\) is stratified, we have a bijection \(Spc^h(\mc T_1^c) \simeq Spc(\mc T_1^c)\) (see \cite[Theorem~4.7]{comparisonsupports}). Moreover, under the equivalence \(F: \mc T_1^c \to \mc T_2^c\) we also obtain \(Spc^h(\mc T_2^c) \simeq Spc(\mc T_2^c)\). Fix a \(\mc P \in Spc(\mc T_2^c)\) and let \(\beta_\mc P \in Spc^h(\mc T_2^c)\) be the corresponding homological prime. It follows from \cite[Lemma~3.7]{comparisonsupports} that \(g(\mc P) \otimes E_{\beta_\mc P} \ne 0 \), and \(g(\mc P) \otimes E_\beta = 0\) for all other \(\beta\). Then we have \begin{align*}
        Locid(g(\mc P) ) & = Locid(g(\mc P) ) \otimes \mc T_2\\
         & = Locid(g(\mc P) ) \otimes Locid(E_\beta: \beta \in Spc^h(\mc T_2^c)) \hspace{.3cm} \\
         & = Locid(g(\mc P) \otimes E_\beta : \beta \in Spc^h(\mc T_2^c))\\
         & = Locid(g(\mc P) \otimes E_{\beta_\mc P} )\\
         & \subseteq Locid(E_{\beta_\mc P}).
    \end{align*}
    Since \(E_{\beta_\mc P}\) is a field object it generates minimal localizing ideals. Indeed we have \[0 \ne \coprod_{j \in I} \Sigma^{m_j} E_{\beta_\mc P} \simeq g(\mc P) \otimes E_{\beta_\mc P} \in Locid(g(\mc P))\]
    Since localizing ideals are thick, \(Locid(g(\mc P))\) contains \(E_\beta\) which gives the reverse containment, \( Locid (E_{\beta_\mc P}) \subseteq Locid(g(\mc P))\). Note that this also proves that the localizing ideals \(Locid(g(\mc P))\) are minimal. Finally we note that, by Proposition \ref{E_b generate cat} \begin{align*}
        \unit \in &Locid(E_\beta: \beta \in Spc^h(\mc T^c)) = Locid(g(\mc P): \mc P \in Spc(\mc T^c) ),
    \end{align*} which shows the local-to-global principle also holds. Hence \(\mc T_2\) is stratified. 
\end{proof}
\section{Minimality at the Height Two Prime}\label{m_psection}
In this section we deduce minimality for the last remaining prime \(m_p\) in \(DTM(\olQ,\ZZ_{(p)})\). This allows us to conclude that \(\dtmz\) is stratified.
\begin{remark}\label{reductionremarkformp}
    We again consider the geometric functor \[\gamma^*_p : \dtmp \to DTM(\olQ, \ZZ/p)\]
induced by the ring map \(\ZZ_{(p)} \to \ZZ/p\). On spectra, we get the following picture 
\[\begin{tikzcd}
	{Spc(DTM(\olQ, \ZZ/p)^c)} & {Spc(\dtmp^c)} \\
	{\MM: =} {(0)} & {m_p} \\
	{\mc P: =} { \langle cone(\beta _p) \rangle } & {e_p} \\
	&  {m_0}
	\arrow[no head, from=2-1, to=3-1]
	\arrow[no head, from=2-2, to=3-2]
	\arrow[no head, from=3-2, to=4-2]
	\arrow[hook, from=1-1, to=1-2]
	\arrow[hook, from=3-1, to=3-2]
\end{tikzcd}\]
\noindent We again have that \(\gamma_*(\unit_{DTM(\olQ, \ZZ/p)}) \in \dtmp^c\), so by Proposition~\ref{compacttheorem} to prove minimality holds at the height 2 prime \(m_p\), we are reduced to proving minimality holds at the unique closed point \(\MM \in Spc(DTM(\olQ, \ZZ/p)^c)\).
\end{remark}
\begin{remark}
    The category of Tate motives with finite coefficients has another useful characterization as the derived category of filtered vector spaces, which we now describe. Let us denote  \(\ZZ^{op}Mod(\ZZ/p)\) to be the category of presheaves on the poset category \(\ZZ\) with coefficients in the category \(Mod(\ZZ/p)\) of \(\ZZ/p\)-modules. A presheaf \(M \in \ZZ^{op}Mod(\ZZ/p) \) is called a \emph{filtered-module} if \(M_{n,n+1}\) is a monomorphism for all \(n\). This is a quasi-abelian category, in the sense of \cite{quasiabaleiancat}, and so it can be derived, which we will denote by \(D_{fil}(\ZZ/p)\). Details of its construction can be found in \cite[\S3]{filteredmodules}. Importantly for us, it is rigidly-compactly generated; see \cite[Corollary~3.4]{filteredmodules}. Moreover, we have the following two equivalences:
\end{remark}
\begin{theorem}\cite[3.16]{derivedcatoffilterednotmartin} \label{derivedequivfilteredandpresheaf}
There is an equivalence of tt-categories between \(D_{fil}(\ZZ/p)\) and \(D(\ZZ^{op}Mod(\ZZ/p))\).
\end{theorem}
\begin{proposition}\cite[Propositions~7.7,~7.9]{martinspecpaper}
There is an equivalence of triangulated categories \[pos: D_{fil}(\ZZ/p)^c \xrightarrow{\sim}  DTM(\olQ, \ZZ/p)^c\]that induces a bijection of thick tensor ideals. 
\end{proposition} 
\begin{remark}
   The proof of this last proposition uses the existence of a strongly hereditary t-structure on \(D_{fil}(\ZZ/p)^c\) \cite[Lemma~7.6]{filteredmodules}. Moreover, to obtain the bijection of thick tensor ideals between the two categories, Gallauer only needed to show the functor is tensor on a certain subcategory of \(D_{fil}(\ZZ/p)^c\). In particular, the functor is tensor on the heart of the t-structure. Since the t-structure is strongly hereditary, and generates \(D_{fil}(\ZZ/p)^c\), every object is a finite sum of shifts of objects in the heart, see for example \cite{heridtstructure}. Using this, one can show that the functor is really tensor everywhere.
\end{remark}
\begin{corollary}
    The functor \(pos: D_{fil}(\ZZ/p)^c \to  DTM(\olQ, \ZZ/p)^c \) is a tensor triangulated equivalence.
\end{corollary}
\begin{remark}\label{underivedequivgraded}
    Let \(M \in \ZZ^{op}Mod(\ZZ/p)\). Associated to \(M\) is the graded \(\ZZ/p[\beta]\)-module \(\bigoplus_{n \in \ZZ} M_n\) where \(\beta\) has degree \(-1\) and acts by~\(\beta: M \to M(1)\). Conversely, given a graded \(\ZZ/p[\beta]\)-module \(\bigoplus_{n \in \ZZ}M_n\) we get a presheaf sending \(n\) to \(M_n\) with transition maps given by \(\beta: M_n \to M_{n-1}\). This provides a tensor-equivalence between the Grothendieck categories \(\ZZ^{op}Mod(\ZZ/p)\) and \(Mod_{gr}(\ZZ/p[\beta])\), see for example \cite[Lemma.~2.2]{derivedcatofgradedring} and the references therein. 
\end{remark}
\begin{remark}
    Hence, by combining Theorem~\ref{derivedequivfilteredandpresheaf} and Remark~\ref{underivedequivgraded}, we obtain a tt-equivalence between \(D_{fil}(\ZZ/p)\) and \(D(Mod_{gr}(\ZZ/p[\beta]))\). The tt-geometry of graded \(\ZZ/p [\beta]\)-modules has been studied by Stevenson and Dell'Ambrogio, and more recently by Barthel, Heard and Sanders. Let us summarize what is needed for us below:
\end{remark}
\begin{theorem}\cite{derivedcatofgradedring, comparisonsupports} \label{derivedcatofgradedringtheorem} 
The category \(\mc T:= D(Mod_{gr}(\ZZ/p[\beta]))\) is stratified. As a consequence, \(D_{fil}(\ZZ / p)\) is stratified as well. Moreover, we have that:\begin{enumerate}
    \item [(1)] \(\mc T ^c\) is countable, and hence satisfies Brown--Adams representability.
    \item [(2)] The natural map \(\varphi: Spc^h(\mc T^c) \xrightarrow{\sim} Spc(\mc T^c)\) is a homeomorphism. Moreover, the `naive' homological support \[Supp^{h}_{\text{naive}}(t) : = \{\beta \in Spc^h(\mc T^c) : E_\beta \otimes t \ne 0\}\]
    corresponds with the actual homological support.
    \item [(3)] Letting \(p\) denote the unique prime corresponding to \(\beta \in Spc^h(\mc T^c)\), the homological primes \(E_{\beta} \simeq k(p)\) are field objects in \(\mc T\).
    \item [(4)] For all nonzero \(t \in \mc T\) there exists a non-zero map from \(E_\beta \otimes x \to t\) for some \(\beta \in Spc^h(\mc T^c)\) and \(x \in \mc T^c.\)
\end{enumerate}
\end{theorem}
\begin{proof}
That \(\mc T\) is stratified is \cite[Theorem~5.7]{derivedcatofgradedring}. Part \((2)\) is a general consequence of stratification, see \cite[Theorem~4.7]{comparisonsupports} and the remark after it. Part \((3)\) is \cite[Example~5.3]{comparisonsupports} and \cite[Lemma~4.2]{derivedcatofgradedring}. Finally, part \((4)\) is \cite[Proposition~4.7]{derivedcatofgradedring}.
\end{proof}
\begin{proposition}
    Continue to let \(\mc T:= D(Mod_{gr}(\ZZ/p[\beta]))\) and denote by \(\mc A = Add((\mc T^c)^{op}, \mc Ab)\) the category of modules on \(\mc T\) as in \S\ref{brownadamssection}. Then for all non-zero homological functors \(0 \ne \hat{t} \in \mc A\) there exists a non-zero map \(\hat{E}_\beta \otimes \hat{x} \to \hat{t}\) for some \(\beta \in Spc^h(\mc T^c)\) and \(x \in \mc T^c\).
\end{proposition}
\begin{proof}
    As the corresponding statement holds for \(\mc T\) by Theorem \ref{derivedcatofgradedringtheorem}, we really have to just show there are no phantom maps out of \(E_\beta \otimes x\) for any \(\beta \in Spc^h(\mc T^c)\) and \(x \in \mc T^c\). To do so, first note that the graded ring \(R: = \ZZ/p[\beta]\) is a graded-local, regular, noetherian ring (see \cite[\S1.5]{gradedregularrings} for example). Letting \(\mathfrak{m}\) denote the unique maximal ideal in \(R\), we have \(E_{\beta_\MM} \simeq k(\mathfrak{m}) \in \mc T^c\) because \(R\) is graded-regular. Now let \(0 \ne t \in \mc T\) be arbitrary. If there is a nonzero map \(f: E_{\beta_\MM} \otimes x \to t\) for some \(x \in \mc T^c\), then the map \(\hat{f}: \hat{E}_{\beta_\MM} \otimes \hat{x} \to \hat{t}\) remains nonzero since there can be no phantom maps out of compact objects. In this case the proof is complete.

    Otherwise, assume there are no-nonzero maps \(E_{\beta_\MM} \otimes x \to t\) for any compact object \(x \in \mc T^c\). This implies that \([E_{\beta_\MM}, t] = 0\), which, since \(E_{\beta_\MM} \in \mc T^c\), tells us that \([E_{\beta_\MM}, \unit] \otimes t = 0\). Since \(E_{\beta_\MM}\) is a direct summand of \(E_{\beta_\MM} \otimes [E_{\beta_\MM}, \unit] \otimes E_{\beta_\MM}\) this forces \(E_{\beta_\MM} \otimes t = 0\) and so \(E_{\beta_\MM} \notin Supp^h(t)\). Again using that \(\mc T\) is stratified, we conclude that \(\MM \notin Supp(t)\). Now recall from Remark \ref{ g(P) objects in two-point case} that \(g(\MM) = e_{\MM}\) and \(g(\mc P) = f_{\MM}\), and so we get \(e_{\MM} \otimes t = 0\), and \(t \simeq t \otimes f_{\MM}\). Let \(f^*: \mc T \to \mc T_{\mc P}\) be the finite localization associated to \(\mc P\). Recall that \(f_*\) is fully faithful, and that \(f_*(\unit) = f_{\MM}\). Then using the projection formula we get \(f_*f^*(t) = t \otimes f_{\MM}  \simeq t\). Moreover, Gallauer showed that \(\mc T_\mc P \simeq D(\ZZ/p)\) as tensor-triangulated categories (see \cite[Lemma~3.7, Lemma~5.3]{filteredmodules}). Indeed, when thinking of \(\mc T\) as the category of filtered \(\ZZ/p\)-vector spaces, the functor \(f^*: \mc T \to \mc T_\mc P\) corresponds to forgetting the filtration. Now, by assumption we know that there is a nonzero map \(\alpha: E_{\beta_\mc P} \otimes x \to t\) for some compact \(x\). Our goal is to show that \(\hat{\alpha}\) remains nonzero in \(\mc A\). To do so, consider the following commutative diagram: \[\begin{tikzcd}
	\mc T & \mc A \\
	\mc T_\mc P & \mc A_{\mc P}
	\arrow["h", from=1-1, to=1-2]
	\arrow["{f^*}"', from=1-1, to=2-1]
	\arrow["{\hat{f}^*}", from=1-2, to=2-2]
	\arrow["h"', hook, from=2-1, to=2-2]
\end{tikzcd}\]
Since \(\mc T_\mc P\) is the derived category of a field, it is phantomless, and so to show that \(\hat{\alpha}\) is nonzero in \(\mc A\), by the commutativity of the above diagram, it suffices to show that \(f^*(\alpha) \ne 0\). To do so, let \(\eta\) denote the counit of the adjunction \(f^* \dashv f_*\). Recall that since \(f_*\) is fully faithful \(\eta\) is an isomorphism. Finally, consider the following commutative diagram: \[\begin{tikzcd}
	{E_{\beta_\mc P} \otimes x} & t \\
	{f_*f^*(E_{\beta_\mc P} \otimes x)} & {f_*f^*(t)}
	\arrow["\alpha", from=1-1, to=1-2]
	\arrow["\eta"', from=1-1, to=2-1]
        \arrow["\simeq", from=1-1, to=2-1]
	\arrow["\eta", from=1-2, to=2-2]
        \arrow["\simeq"', from=1-2, to=2-2]
	\arrow["{f_*f^*(\alpha)}"', from=2-1, to=2-2]
\end{tikzcd}\]
This implies that \(f_*f^*(\alpha) \ne 0\) and so \(f^*(\alpha) \ne 0\) as desired.
\end{proof}
This allows us to use the results from \S\ref{brownadamssection} to immediately conclude: 
\begin{theorem}\label{dtmmodpstrattheorem}
The category \(DTM(\olQ, \ZZ/p)\) is stratified.
\end{theorem}
\begin{proof}
We have a tt-equivalence \(D_{fil}(\ZZ/p)^c \simeq DTM(\olQ, \ZZ/p)^c\) where \(D_{fil}(\ZZ/p)\) is stratified, the subcategory of compact objects is countable, and whose homological primes \(E_\beta\) are field objects by Theorem \ref{derivedcatofgradedringtheorem}. The last proposition gives us the final assumption needed to apply Theorem \ref{stratifiedusingbrownadamstheorem} to conclude that \(DTM(\olQ, \ZZ/p)\) is stratified as well.
\end{proof}
As a result, keeping in mind Remark \ref{reductionremarkformp}, we conclude that:
\begin{corollary}\label{minatmp}
    Minimality holds at the height 2 prime \(m_p\).
\end{corollary}
\section{Stratification of Tate Motives}
Let us bring the results from Sections~\ref{m_psection}, \ref{e_psection}, and \ref{m_0section} together into our main theorem. 
\begin{theorem}[Stratification]\label{maintheoremdtmstrat}
    The category \(DTM(\olQ, \ZZ)\) is stratified.
\end{theorem}
\begin{proof}
    Since \(DTM(\olQ, \ZZ)\) satisfies the local-to-global principle automatically because \(Spc(DTM(\olQ, \ZZ)^c)\) is noetherian, we just have to show minimality holds at each prime. This is precisely what we showed in Corollaries~\ref{minatmp},~\ref{minatep}, and~\ref{minatmo}.
\end{proof}
As remarked in the introduction, an important consequence of a category being stratified is an affirmative answer to the abstract \textit{telescope conjecture} \cite[Theorem~9.11]{BHSStrat}. Hence as a corollary to the above theorem we immediately get:
\begin{theorem}\label{telescopeconjecture}
    Every smashing ideal of \(DTM(\olQ, \ZZ)\) is compactly generated. That is, every smashing localization is a finite localization.
\end{theorem}
\bibliographystyle{alpha}
\bibliography{citations}
\end{document}